\documentclass{amsart}
\usepackage{amsmath, amsfonts, amsthm, amssymb}
\usepackage{ifthen}
\usepackage{amsmath}
\usepackage{amssymb}
\usepackage{latexsym}
\usepackage{bbm}
\usepackage{graphicx}
\usepackage{verbatim}
\usepackage{fancyhdr}
\usepackage{url}               
\usepackage[applemac]{inputenc}

\newcommand{\rr}{\mathbb{R}}

\newcommand{\cc}{\mathbb{C}}

\newcommand{\eps}{\epsilon}

\newcommand{\pl}{\partial}
\newcommand{\x}{\times}

\newcommand{\til}{\widetilde}

\newcommand{\cjd}{\rangle}
\newcommand{\cjg}{\langle}

\newcommand{\demi}{\frac{1}{2}}

\newcommand{\la}{\lambda}

\newcommand{\be}{\begin{equation}}
\newcommand{\ee}{\end{equation}}
\newcommand{\ba}{\begin{aligned}}
\newcommand{\ea}{\end{aligned}}
\newcommand{\bsp}{\begin{split}}
\newcommand{\esp}{\end{split}}
\newcommand{\bp}{{\it Proof. }}
\newcommand{\ep}{\hfill $\square$\\}
\newcommand{\N}{\mathbb N}
\newcommand{\R}{\mathbb R}

\newcommand\dist{\textup{dist}}

\newcommand\bT{\overline T}
\newcommand{\indic}{\operatorname{1\negthinspace l}}

\def\<{\langle}
\def\>{\rangle}

\newtheorem{lem}{Lemma}[section]
\newtheorem{thm}[lem]{Theorem}
\newtheorem{prop}[lem]{Proposition}

\newtheorem{remk}[lem]{Remark}

\numberwithin{equation}{section}
\numberwithin{figure}{section}

\begin{document}

\title{Spectral analysis of random walk operators on euclidian space}
\author[C. Guillarmou]{Colin~Guillarmou}
\email{cguillar@\allowbreak dma.\allowbreak ens.\allowbreak fr}

\author[L. Michel]{Laurent~Michel}
\email{lmichel@\allowbreak unice.\allowbreak fr}

\maketitle
\begin{abstract}
We study the operator associated to a random walk on $\R^d$ endowed with a probability measure. We give a precise description of the spectrum of the operator near $1$ and use it to estimate the  total variation distance between the iterated kernel and its stationary measure. Our study contains the case of Gaussian densities on $\R^d$.
\end{abstract}
\section{Introduction}

Let $\rho\in C^1(\R^d)$ be a strictly positive bounded function such that $d\mu=\rho(x)dx$ is a probability measure. 
Let $h>0$ be a small parameter and $B_h(x)$ be the ball of radius $h$ and center $x$. We consider the natural random walk associated to the density $\rho$ with step $h$:
if the walk is in $x$ at time $n$, then the position $y$ at time $n+1$ is determined by chosing $y\in\R^d$ uniformly with respect  to the measure 
\be
t_h(x,dy)=\frac {\rho(y)}{\mu(B_h(x))}\indic_{\vert x-y\vert<h}dy
\ee
The associated random-walk operator is defined by
\be
T_hf(x)=\frac 1{\mu(B_h(x))}\int_{B_h(x)}f(x')d\mu(x').
\ee
for any continuous function  $f$, and
the kernel of $T_h$ is $t_h(x,dy)$.
This is clearly a Markov kernel. Introduce the measure
 \[d\nu_h=\frac{\mu(B_h(x))\rho(x)}{Z_h}dx\] 
 where $Z_h$ is chosen so that $d\nu_h$ is a probability on $\R^d$. Then, the operator $T_h$ is self-adjoint on $L^2(M,d\nu_h)$ and the measure $d\nu_h$ is stationnary for the kernel $t_h(x,dy)$ (this means that $T_h^t(d\nu_h)=d\nu_h$, where $T_h^t$ is the transpose operator of $T_h$ acting on Borel measures).

The aim of this article is to describe the spectrum ot $T_h$ and to adress the problem of convergence of the iterated operator to the stationary measure. Such problems have been investigated in compact cases in \cite{DiaLeb}, \cite{LebMi} and \cite{DiaLebMi}, and the link between the spectrum of $T_h$ and the Laplacian (with Neumann boundary condition in \cite{DiaLeb} and \cite{DiaLebMi}) was etablished. In this paper we investigate the case of such operators on the whole Euclidian space. The main difference with the previous works comes from the lack of compactness due to the fact that $\R^d$ is unbounded.
We will make the following assumptions on $\rho$:\\ 

\noindent\textbf{Case 1: tempered density}. A density $\rho\in C^1(\rr^d)$ 
is said \emph{tempered} if there exists a constant $C>0$ such that for all $x\in \rr^d$
\be\label{boundderiv}
|d \rho(x)|\leq C\rho(x)
\ee
We shall say that it is \emph{smooth tempered of exponential type} if $\rho$ is smooth and
if there are some positive numbers $ (C_\alpha)_{\alpha\in \N^d} $, $R>0, \kappa_0>0$, such that
\be\label{hyp1}
\forall \vert x\vert\geq R,\;\vert\partial_x^\alpha\rho(x)\vert\leq C_\alpha \rho(x)
\ee
and, if $\Delta:=-\sum_{i=1}^d\pl_{x_i}^2$ is the positive Laplacian,
\be\label{hyp2}
\forall \vert x\vert\geq R,\;-\Delta\rho(x)\geq \kappa_0\rho(x).
\ee
Densities verifying these assumptions can be easily  constructed. For instance, if $\rho$ is a smooth non vanishing function such that 
there exists $\alpha,\beta>0$ such that for any $\vert x\vert>R$ we have 
$\rho(x)=\beta e^{-\alpha\vert x\vert}$, then the above assumptions are satisfied with $\kappa_0=\alpha^2$. 
For  densities satisfying \eqref{hyp1}, \eqref{hyp2}, we will define
\be\label{eq:def_kappa}
\kappa=\lim_{R\to \infty}\inf_{|x|\geq R}\frac{-\Delta\rho(x)}{\rho(x)}.
\ee 

The second type of densities we shall consider is the following\\

\noindent{\bf Case 2: Gaussian density}. We assume that $\rho(x)=\beta e^{-\alpha |x|^2}$ for some $\alpha,\beta>0$ such that 
$\int_{\R^d}\rho(x)dx=1$.\\

It can be shown that  that if  $\rho$ satisfies \eqref{boundderiv} or is Gaussian, 
there exists a constant $C>0$ and $h_0>0$ such that 
\be\label{hyp0}
\forall x\in\R^d,\,\forall h\in]0,h_0],\;\mu(B_h(x))\geq Ch^d\rho(x).
\ee
Let us set $m_h(x)=\mu(B_h(x))$ and define the functions
\be\label{eq:def_ah}
a_h(x):=(\alpha_dh^{d}\rho(x)/m_h(x))^{1/2},\quad 
G_d(\xi)=\frac 1{\alpha_d}\int_{\vert z\vert \leq 1}e^{iz\xi}dz
\ee 
where $\alpha_d:={\rm Vol}(B_{\R^d}(0,1))$. Notice that $G_d$ is a real valued function bounded above by $1$ and below by 
some $M>-1$, then define 
\be\label{limsup}
A_h:=\lim_{R\to \infty}\sup_{|x|\geq R}a_h^2(x), \quad M:=\min_{\xi\in \rr^n} G_d(\xi)>-1. 
\ee 
We will show that $A_h=1-\frac{\kappa}{2(d+2)}h^2+O(h^4)$ with 
$\kappa$ defined in \eqref{eq:def_kappa}.

In order to describe the eigenvalues of $T_h$, let us also introduce the operator 
\be\label{eq:def_Lrho}
L_\rho=\Delta+V(x)
\ee
with $V(x):=\frac{-\Delta\rho(x)}{\rho(x)}$. 
Observe that the essential spectrum of this operator is $[\kappa,+\infty[$. Moreover, we have the following factorisation:
\be
L_\rho=\sum_{j=1}^d\ell_j^*\ell_j
\ee
where $\ell_j=-\partial_{x_j}+\frac{\partial_{x_j}\rho}{\rho}$.
This shows that $L_\rho$ is non-negative on $L^2(\R^d)$. Moreover, since $\ell_ju=0$ iff $u$ is proportional to $\rho$, then  $0$ is a simple eigenvalue associated to the eigenfunction $\rho\in L^1\cap L^\infty\subset L^2$.

We first prove the following result in the tempered case
\begin{thm}\label{th:analyse_spec_temp}
Suppose that $\rho$ is tempered in the sense of \eqref{boundderiv}, then:\\ 
(i) the essential spectrum of $T_h$ on $L^2(\rr^d,d\nu_h)$
is contained in $[A_hM, A_h]$ where $M$ and $A_h$ are defined in \eqref{limsup}. 
If in addition $A_h=\lim_{|x|\to \infty}a^2_h(x)$, then 
$\sigma_{\rm ess}(T_h)=[A_hM, A_h]$.\\
\noindent (ii) If \eqref{hyp1} and \eqref{hyp2} hold,  
then $A_h=1-\frac{\kappa}{2(d+2)}h^2+O(h^4)$ with $\kappa$ defined in \eqref{eq:def_kappa}, and 
for all $\alpha\in]0,1[$ there exist $C>0,h_0>0$ such that, if  $0=\mu_0<\mu_1\leq \mu_2\leq\ldots\leq\mu_k$ denote 
the $L^2(\rr^d,dx)$ eigenvalues  of $L_\rho$ in $[0,\alpha\kappa]$ counted with multiplicities, 
and if $1=\la_0(h)>\la_1(h)\geq \ldots\geq\la_k(h)$ denote the $k$ largest  
eigenvalues of $T_h$ on $L^2(\rr^n,d\nu_h)$ counted with multiplicities, 
then for all 
$h\in]0,h_0]$ and any $j=1,\ldots,k$,
\[
\Big\vert 1-\frac{1}{2(d+2)}\mu_kh^2-\lambda_k(h)\Big\vert\leq Ch^4.
\]
\end{thm}

Observe that if $\rho$ is only tempered, the statement (i) shows that the essential spectrum can be the whole interval $[M,1]$:
for instance, take a density $\rho$ such that $\rho(x)=|x|^{-m}$ in $\{|x|>R\}$ for some $R>0$ and $m>d$, then
it is easy to check that $m_h(x)/\rho(x)\to 1$ as $|x|\to \infty$ and therefore $A_h=1$ in this case.
 
Notice also that there are examples of smooth densities of exponential type $\rho$ such that the discrete spectrum 
of $L_\rho$ below its continuous spectrum is non-empty. Indeed, take for instance 
$\rho=e^{-\tau \alpha(x)}$ where $\tau>0$ and $\alpha(x)$ is smooth, equal to $\vert x\vert$
for $|x|>1$ and $\alpha(0)=0$, then 
\[P_\tau:=\tau^{-2}L_\rho=\tau^{-2}\Delta +|\nabla\alpha|^2+\tau^{-1}\Delta \alpha\]
is a $\tau^{-1}$ semi-classical elliptic differential operator with semi-classical principal symbol
$p(x,\xi)=|\xi|^2+|\nabla \alpha|^2$ (see \cite{Ma,EvZw,DiSj99} for the theory of 
semi-classical pseudodifferential operators). 
Since $|\nabla\alpha|=1$ in $|x|>1$ and $\Delta\alpha=0$ in $|x|>1$, the essential spectrum of $P_\tau$
is $[1,\infty)$, then we can apply Theorem 9.6 of \cite{DiSj99} and the fact that ${\rm Vol}\{(x,\xi)\in \rr^{2d}; p(x,\xi)\in [0,\demi]\}>0$ (since $\alpha(0)=0$)
to conclude that, if $\tau>0$ is large enough, 
there exist $C\tau^{d}$ eigenvalues of $P_\tau$ in $[0,\demi]$  for some $C>0$.

We also emphasize that the result in Theorem \ref{th:analyse_spec_temp} 
is used in a fundamental way in the recent paper \cite{CGM} to analyze random walks
on surface with hyperbolic cusps.\\

If instead $\rho$ is Gaussian, then $L_\rho=\Delta+4\alpha^2|x|^2-2d\alpha$ and its spectrum is discrete 
$\sigma(L_\rho)=4\alpha\N$ and the  eigenfunctions  associated to $4\alpha k$ 
have the form $H_{k}(x)e^{-2\alpha\vert x\vert^2}$ for some explicit polynomial $H_{k}$.
We then have 
\begin{thm}\label{th:analyse_spec_gauss}
Suppose that $\rho$ is Gaussian, then the operator $T_h$ is compact and  
if $0=\mu_0<\mu_1\leq \mu_2\leq\ldots\leq\mu_k\ldots$ denote the $L^2(\rr^d,dx)$ eigenvalues of $L_\rho$ and 
$1=\la_0(h)>\la_1(h)\geq \dots \la_{k}(h)\geq \dots $ those of $T_h$,  
then for $K\geq 0$ fixed, there exists $C>0$ and $h_0>0$ such that for all $h\in]0,h_0]$ and any $k=1,\ldots,K$,
\be\label{eq:th:loc_vp}
\Big\vert 1-\frac{1}{2(d+2)}\mu_k h^2-\lambda_k(h)\Big\vert \leq Ch^4.
\ee
Moreover, there exists $\delta_0>0$ such that for any $\lambda\in [0,\delta_0]$, the number $N(\lambda,h)$ of eigenvalues of $T_h$ in $[1-\lambda,1]$ satisfies
\be\label{eq:weyl}
N(\lambda,h)\leq C(1+\lambda h^{-2})^{d }.
\ee
\end{thm}

In the last section of this paper, we also give some consequences on the convergence of the kernel of $T_h^n$ to the stationary measure $d\nu_h$ as $n\to\infty$.  In particular we show that, contrary to the compact setting \cite{LebMi}, the convergence in $L^\infty$ norm fails, 
essentially due to the non-compactness of the space.\\ 

These theorems, will be proved by using microlocal analysis. We refer to the books, \cite{DiSj99}, \cite{EvZw} and \cite{Ma} for standard results in this theory. The organisation of the paper is the following.
In the next section we study  the essential spectrum of $T_h$ on $L^2(\rr^d,d\nu_h)$. In section 3, we collect some a priori estimates (regularity and decay) on the eigenfunctions of $T_h$. Following the strategy of \cite{LebMi}, we use these estimates in section 4 to prove the above theorems.
In last section, we adress the problem of total variation estimates: we show that the convergence to stationarity can not be uniform with respect to the starting point. Considering the case where the starting point $x$ belongs to a ball of radius $\tau$ we prove total variation bounds in term of the spectral gap and $\tau$.

\section{Essential spectrum}

We start by studying the essential spectrum of $T_h$ in the tempered and Gaussian cases.
>From the definition of $d\nu_h$, it is easy to see that there exists some constant $c_1,c_2>0$ such that $ c_1h^d\leq Z_h\leq c_2h^d$.
Let us define the operator $\Omega:L^2(\R^d,dx)\rightarrow L^2(\R^d,d\nu_h)$ by
\be\label{defconj}
\Omega f(x)=\sqrt{\frac{Z_h}{m_h(x)\rho(x)}}f(x).
\ee
which is unitary, and let $\tilde T_h$ defined by $\tilde T_h=\Omega^*T_h\Omega$ so that
\[
\tilde T_hf(x)=a_h(x)\bT_h(a_hf)
\]
with $a_h$ defined in \eqref{eq:def_ah} and (with $\alpha_d={\rm Vol}(B_{\rr^d}(0,1))$)
\be\label{bT}
\bT_hg(x):=\frac 1 {\alpha_dh^{d}}\int_{\vert x-y\vert <h}f(y)dy.
\ee
Using the semiclassical Fourier transform it is easy to see that 
\[\bT_h=G_d(hD_x),  \textrm{ with } G_d(\xi)=\frac 1{\alpha_d}\int_{\vert z\vert \leq 1}e^{iz\xi}dz.\]
This function depends only on $|\xi|$, it is clearly real valued and $ -1<M\leq G_d(\xi)\leq 1$ for all $\xi$ if 
$M$ is defined in \eqref{limsup}. Moreover, $G_d$ tends to zero at infinity and $G_d(\xi)=1$ if and if only $\xi=0$.

Let us first prove \eqref{hyp0} assuming \eqref{boundderiv}: we have by assumption on $\rho$ that for all $x,y\in\rr^d$
with $|x-y|\leq h$
\[ -Ch\sup_{z\in B_h(x)}\rho(z)\leq \rho (x)-\rho(y)\leq Ch\sup_{z\in B_h(x)}\rho(z)\]
and therefore if $Ch<1$
\[  (1-Ch)\sup_{z\in B_h(x)}\rho(z)  \leq   \rho(x) \leq \sup_{z\in B_h(x)}\rho(z)\]
which implies  
\[  \rho(x)\Big(1-\frac{Ch}{1-Ch}\Big)\leq \rho(y)\leq \rho(x)\Big(1+\frac{Ch}{1-Ch}\Big).\]
and thus \eqref{hyp0}. 
 
The function $a_h$ is then bounded and $A_h$ of \eqref{limsup} is well defined. 
We first prove 
\begin{prop}\label{specess}
Suppose that $\rho$ is tempered in the sense of \eqref{boundderiv}, then $\sigma_{\rm ess}(T_h)\subset [MA_h,A_h]$. 
If moreover $A_h=\lim_{|x|\to \infty}a^2_h(x)$, then the inclusion above is an equality.
\end{prop}
\begin{proof}
Let  $R>0$, then the operator $\tilde T_h$ can be written under the form 
\[ \tilde T_h= \indic_{|x|>R}\tilde T_h\indic_{|x|>R}+ \indic_{R<|x|<R+h}\tilde T_h\indic_{|x|<R}+\indic_{|x|<R}\tilde T_h\indic_{R<|x|<R+h}
\]
since $\tilde T_h$ increases support by a set of diameter at most $h$. 
The kernels of the last two operators in the right hand side
is in $L^2(\R^d\x\R^d, dx\otimes dx)$, and thus these operators are compact. We thus deduce that the essential spectrum of $\tilde T_h$ is given by that of $S_h^R=\indic_{|x|>R}\tilde T_h\indic_{|x|>R}$.
Since $S_h^R=b_h^R\overline{T}_hb_h^R$ with $b_h^R=\indic_{|x|>R}a_h(x)$ and since $\bT_h$ is a bounded self-adjoint operator satisfying 
\[  M||f||^2_{L^2}\leq \cjg \bT_hf,f\cjd_{L^2}, \quad ||\bT_hf ||_{L^2} \leq ||f||_{L^2}\]
and $a_h(x)>0$ we deduce easily that  $\sigma_{\rm ess}(S^R_h)\subset[-MA_h^R,A_h^R]$ 
where $A_h^R:=\sup_{|x|\geq R}a_h(x)^2$. It then suffices to take the limit as $R\to \infty$.
Now if in addition $a^2_h(x)$ has a limit $A_h$ when $|x|\to \infty$, we can write 
\be\label{decomp}
\tilde T_h=A_h \bT_h +\eps_h(x)\bT_ha_h(x)+A_h^\demi\bT_h\eps_h(x)
\ee
with $\eps_h(x):=a_h(x)-A_h^\demi$ converging to $0$ as $|x|\to \infty$. In particular, using that 
$|G_d(\xi)|\to 0$ when $|\xi|\to \infty$, we deduce that the last two operators in \eqref{decomp} are
compact on $L^2$. Since, $\bT_h$ is a function of the Euclidean Laplacian (or radial Fourier multiplier)  
the spectrum of $\bT_h$ on $L^2(\rr^d,dx)$ is absolutely continuous and consists of 
$[M,1]$, which is the range of $G_d(\xi)$. This achieves the proof since the essential spectrum of $\til{T}_h$
is that of $A_h\bT_h$ by \eqref{decomp}. 
\end{proof}

We also describe the asymptotic behaviour of $A_h$:
\begin{lem}\label{asympAh}
If $\rho$ satisfies \eqref{hyp1} and \eqref{hyp2}, then the following asymptotic holds as $h\to 0$
\[A_h=1-\frac{\kappa}{2(d+2)}h^2+O(h^4)\]
where $\kappa=\liminf_{|x|\rightarrow \infty}\frac{-\Delta\rho(x)}{\rho(x)}$.
\end{lem}
\begin{proof}
If $\rho$ is tempered, we expand $m_h(x)=\mu(B_h(x))$ with respect to $h$ and use assumption (\ref{hyp1}):
\[\begin{split}
m_h(x)= &\,  h^d\int_{\vert z\vert <1}\rho(x+hz)dz\\
=& \, \alpha_dh^d\rho(x)+\demi h^{d+2}\sum_{i,j}\pl_{x_i}\pl_{x_j}
\rho (x) \int_{|z|\leq 1}z_iz_j dz+O(h^{d+4}\rho_4)\\
=  &\,\alpha_dh^d\rho(x)-\frac{\beta_d}{2d} h^{d+2}\Delta\rho(x)+O(h^{d+4}\rho_4)  
\end{split}\]
with $\vert\rho_4(x)\vert\leq \rho(x)$ and  $\beta_d:=\int_{\vert z\vert <1}\vert z\vert^2dz$.
Using the definition of $a_h$, it follows from Lemma \ref{lem:prop_ah} below that
\be\label{eq:expand_ah}
a^2_h(x)=1+h^2\gamma_d \frac{\Delta \rho(x)}{\rho(x)}+O(h^4)
\ee 
with $\gamma_d=\frac{\beta_d}{2d\alpha_d}=\frac{1}{2(d+2)}$ and the $O(h^4)$ is uniform in $x\in\rr^d$.
Hence, it follows from \eqref{eq:def_kappa} that
\[
A=\limsup_{|x|\rightarrow \infty}a_h^2(x)=1+\gamma_dh^2\liminf_{|x|\rightarrow \infty}\frac{\Delta\rho(x)}{\rho(x)}+O(h^4)
\]
and the proof is complete.
\end{proof}

\begin{remk} In the tempered case, the operator $\gamma_dL_\rho=\frac 1 {2(d+2)}(\Delta+\frac{-\Delta\rho}{\rho})$ has essential spectrum
contained in $[\frac{\kappa}{2(d+2)},\infty)$. If in addition $\kappa=\lim_{|x|\to \infty}\frac{-\Delta\rho}{\rho}$, then 
the essential spectrum is exactly $\sigma_{\rm ess} (L_\rho)=[\frac{\kappa}{2(d+2)},\infty)$ by Theorem 13.9 of \cite{HisSig}.
\end{remk}

Now for the Gaussian case
\begin{prop}\label{prop:spec_ess_gauss}
If $\rho$ is Gaussian, then $T_h$ is a compact operator.
\end{prop}
\begin{proof}
 The symbol $G_d(\xi)$ of $\bT_h$ is decaying to $0$ as $|\xi|\to 0$, a standard argument shows that 
if $\lim_{|x|\to \infty} a_h(x)\to 0$, then $\bT_ha_h$ is compact on $L^2$.
We write 
\be\label{eq:form_ah}
\frac{m_{h}(x)}{h^d\rho(x)}=\int_{|z|\leq 1}e^{-2hx.z-h^2|z|^2}dz
\ee
and by bounding below this integral by a $dz$ integral on a conic region $-z.x>|z|.|x|/2$, we see that
 it converges to $\infty$ when $|x|\to \infty$, which proves the claim. 
\end{proof}

\begin{remk}
In the Gaussian case, the operator $L_\rho=\Delta+4\alpha^2|x|^2-2d\alpha$ has compact resolvent and discrete spectrum.
\end{remk}
 
\noindent\textbf{Notational convention}: For the following sections, all the tempered densities we shall consider will be  
smooth tempered densities  of exponential type (ie. satisfying \eqref{hyp1} and \eqref{hyp2}), 
and therefore we will abuse notation and just call them tempered. 
 
\section{Spectral analysis of $T_h$}

We recall here some  notations. Let $a=a(x,\xi;h)$ be an $h$-dependent family of $C^\infty(\R^{2d})$ function 
and $m(x,\xi)$ be an order function as in \cite{DiSj99}. We 
say that $a$ belongs to the symbol class $S(m)$ if there exists some $h_0>0$ and constants $C_{\alpha,\beta}$ such that for
any $\alpha,\beta\in \N^d$, any $0<h\leq h_0$
\[\vert\partial_x^\alpha\partial_\xi^\beta a(x,\xi;h)\vert\leq C_{\alpha,\beta}m(x,\xi)\]
 For any $a\in S(m)$, we define ${\rm Op}_h(a)$ by
\[{\rm Op}_h(a)f(x)=\frac{1}{(2\pi h)^d}\int e^{\frac{i(x-y).\xi}{h}}a(x,\xi;h)f(y)dyd\xi\]
The standard theory of such operators is developped in \cite{DiSj99}, \cite{EvZw}, \cite{Ma}.

\subsection{Preliminary estimates}
Let us start by some estimates on the symbols of the operator $\tilde{T}_h$, which will be useful
to study its eigenfunctions.
\begin{lem}\label{lem:prop_Gd}
The function $G_d(\xi)$ belongs to $S(\<\xi\>^{-\max(1,\frac{d-1}2)})$. 
\end{lem}
\bp Suppose first that $d\geq 2$.
It is clear that the function $G_d$ is smooth. When $\vert \xi\vert\geq 1$, 
one has 
\[
\partial_\xi^\beta G_d(\xi)=\frac 1{\alpha_d}\int_{\vert z\vert \leq 1}(i z)^\beta e^{iz\xi}dz.
\]
Let $\chi\in C_0^\infty(B_{\rr^d}(0,1))$ be a radial cut-off equal to $1$ on $B_{\rr^d}(0,\frac 1 2)$.
Then the non-stationary phase theorem shows that 
\[
\int_{\vert z\vert \leq 1}\chi(z)z^\beta e^{iz\xi}dz=O(\vert\xi\vert^{-\infty})
\]
On the other hand,
\[
I_\chi:=\int_{\vert z\vert \leq 1}(1-\chi(z))z^\beta e^{iz\xi}dz
=\int_{\frac 1 2}^1(1-\chi)(r)r^{d-1+|\beta|}\Big(\int_{S^{d-1}}e^{ir\omega\xi}\omega^\beta d\omega\Big) dr
\]
For any $r\geq \frac 1 2$ the phase $\omega\mapsto  \omega\xi$ has only two stationary points:  these points are 
non-degenerate so that the stationary phase theorem implies  $I_\chi=O(\vert\xi\vert^{-\frac{d-1}2})$.
In the case $d=1$, the computation is simpler since $G_d(\xi)=\frac{\sin(\xi)}{\xi}$. We leave it to the reader.
\ep

We will also need the following result on the function $a_h$.
\begin{lem}\label{lem:prop_ah} The function $a_h$ is smooth and the following hold true:
\begin{itemize}
 \item if $\rho$ is tempered, then
\be\label{eq:ah_symb_temp}
\forall\alpha\in\N^d,\;\exists C_\alpha>0,\,\forall h\in]0,1],\,|\partial_x^\alpha a_h(x) |\leq C_\alpha h^2
\ee
and there exists $C>0$ such that 
\be\label{eq:ah_inv_temp}
\forall x\in\R^d,\, |\frac 1 {a^2_h(x)}-1-\frac{h^2}{2(d+2)} \frac{-\Delta\rho}{\rho}|\leq Ch^4
\ee
\item if $\rho$ is Gaussian, then
\be\label{eq:ah_symb_gauss}
\forall\alpha\in\N^d,\;\exists C_\alpha>0,\,\forall h\in]0,1],\,|\partial_x^\alpha a_h(x) |\leq C_\alpha h^{\vert\alpha\vert}.
\ee
and
\be\label{eq:ah_inv_gauss1}
\forall M>0,\exists C_M>0,\,\forall |x|<Mh^{-1},\,\Big|\frac 1 {a^2_h(x)}-1-\frac{(4\alpha^2|x|^2-2d\alpha)}{2(d+2)}h^2\Big|\leq C_M|x|^4h^4,
\ee
\be\label{eq:ah_inv_gauss2}
\exists C,R>0,\forall |x|\geq R,\,\,\frac 1 {a^2_h(x)}\geq \max(1+Ch^2|x|^2, Ce^{h|x|})
\ee
\end{itemize}
\end{lem}
\bp
It follows from (\ref{eq:def_ah}) that $a_h(x)=F\circ g_h(x)$ with $F(z)=z^{-1/2}$ and 
$g_h(x)=\frac{m_h(x)}{\alpha_dh^d\rho(x)}$.
Following the arguments of the proof of Lemma \ref{asympAh}, we have when $\rho$ is tempered
(using $\int_{|z|<1}z_idz=\int_{|z|<1}z_iz_jz_kdz=0$)
\be\label{eq:taylor_gh}
\begin{split}
g_h(x)&=1-\frac{h^2}{2(d+2)}\frac{\Delta\rho}{\rho}+\frac{h^4}{\alpha_d\rho(x)}\int_{|z|<1}\rho_4(x,z)dz\\
&=1-\frac{h^2}{2(d+2)}\frac{\Delta\rho}{\rho}+h^4r_4(x)
\end{split}
\ee
where $\rho_4(x,z)$ is a function which satisfies for all $\alpha\in\N^d$:
$$|\partial_x^\alpha\rho_4(x,z)|\leq C_\alpha \rho(x)$$
 uniformly with respect to $x\in\R^d, \vert z\vert\leq 1$ and
  $r_4(x)$ has all its derivatives uniformly bounded on $\R^d$. 
In particular, for any $\alpha\in\N^d\setminus\{0\}$, $\partial_x^\alpha g_h(x)=O(h^2)$.
Hence,  for $h>0$ small enough,  Fa\`a di Bruno formula combined with (\ref{eq:taylor_gh}) shows that $a_h$  is a smooth bounded function such that
\be\label{eq:estim_deriv_ah}
\forall\alpha\in\N^d\setminus\{0\},\;\partial_x^\alpha a_h(x)=O(h^2).
\ee
This shows that $a_h$ enjoys estimate (\ref{eq:ah_symb_temp}) while
 (\ref{eq:ah_inv_temp}) is a direct consequence of (\ref{eq:taylor_gh}).\\

Suppose now that $\rho(x)$ is Gaussian. 
It follows from (\ref{eq:form_ah}) that 
\[
g_h(x)=\frac 1{\alpha_d}\int_{|z|\leq 1}e^{-2hx.z-h^2|z|^2}dz
\]
Hence, there exists $c_0>0$ such that for all $x\in\R^d$, $h\in]0,1]$, $g_h(x)\geq c_0$. Moreover, for all $\alpha\in\N^d$ we have
\[
\partial^\alpha_xg_h(x)=\frac{1}{\alpha_d}\int_{\vert z\vert<1}(-2hz)^\alpha e^{-2hz.x-h^2|z|^2}dz
\]
so that there exists $C_\alpha>0$ such that
\be\label{eq:estim_gh_gauss}
\forall h\in]0,1],\,\forall x\in\R^d,\,\vert \partial^\alpha_xg_h(x)\vert\leq C_\alpha h^{|\alpha|}\vert g_h(x)\vert
\ee
Using again Fa\`a di Bruno formula, we get easily that $a_h$ is a smooth function such that
for any $\alpha\in\N^d$,
\be\label{eq:FaDiB}
\partial_x^\alpha a_h(x)=\sum_{\pi\in\Pi_{\vert\alpha\vert}} C_{\vert\pi\vert}g_h(x)^{-\frac 1 2(2\vert\pi\vert+1)}\Pi_{B\in\pi}\frac{\partial^{|B|}g_h(x)}{\Pi_{j\in B}\partial x_j}
\ee
where $\Pi_{\vert\alpha\vert}$ denotes the set of all partitions of $\{1,\ldots,\vert\alpha\vert\}$, $\vert\pi\vert$ denotes the number of blocks in the partition $\pi$ and $\vert B\vert$ denotes the cardinal of $B$, and $C_{|\pi|}$ is an explicit constant depending on $|\pi|$. Combining this formula with estimate (\ref{eq:estim_gh_gauss}), we get
\be\label{eq:contr_deriv_symb_gauss}
\vert \partial_x^\alpha a_h(x)\vert\leq \sum_{\pi\in\Pi_{\vert\alpha\vert}} C_{\vert\pi\vert}\vert g_h(x)\vert^{-\frac 1 2(2\vert\pi\vert+1)}\Pi_{B\in\pi}\vert h g_h(x)\vert^{|B|}\leq C\vert a_h(x)\vert h^{|\alpha|}
\ee
which proves (\ref{eq:ah_symb_gauss}).

Let us now prove the estimates on $a_h^{-2}=g_h$. 
The same computation as in the tempered case remains valid if we assume that $|hx.z|$ is bounded, which holds true if $h|x|$ is bounded. This shows
(\ref{eq:ah_inv_gauss1}).
In order to prove  (\ref{eq:ah_inv_gauss2}), we observe that there exist constants $c,C>0$ such that for all $0<h<1$ 
\[
\begin{split}
a_h(x)^{-2}=&\alpha_d^{-1}\int_{S^{d-1}}\int_{0<r\leq 1}e^{-2hr x.\theta-h^2r^2}r^{d-1}drd\theta\\
=&\alpha_d^{-1}\int_{S^{d-1}}\int_{0<r\leq 1}\Big(1+4r^2h^2(x.\theta)^2\int_0^1 e^{-2t h rx.\theta}\frac{(1-t)}{2}dt\Big)e^{-h^2r^2}r^{d-1}drd\theta\\
\geq &1+4h^2\alpha_d^{-1}\int_{S^{d-1}}\int_{0<r\leq 1} r^2(x.\theta)^2\Big(\int_0^1 e^{-2t h rx.\theta}\frac{(1-t)}{2}dt\Big)e^{-h^2r^2}r^{d-1}drd\theta\\
& -ch^2\\
a_h(x)^{-2}\geq & 1+ Ch^2|x|^2-ch^2 
\end{split}\]
for $|x|>R$ with $R>0$ large, the last inequality being proved by the same argument as for Proposition \ref{prop:spec_ess_gauss}. 
Enlarging $R$ and modifying $C>0$ if necessary, this shows the quadratic bound in (\ref{eq:ah_inv_gauss2}). The exponential bound in  (\ref{eq:ah_inv_gauss2}) follows easily from the inequality above, by bounding below the integral by an integral 
on a region $\{\theta.x/|x|<-(1-\eps), r\geq 1-\eps\}$ for some small $\eps>0$. 
\ep

\subsection{Regularity and decay of eigenfunctions}
We are now in position to prove the first estimates on the eigenfunctions of $\tilde T_h$. 

Observe that for any $1/2>\delta>0$ small, there exists $s_\delta>0$ such that 
 $|G_d(\xi)|\leq 1-2\delta$ when $|\xi|^2\geq s_\delta$.
\begin{lem}\label{lem:regularite}
Let  $C>0$ and $\lambda_h\in [1-Ch^2,1]$ be an eigenvalue of $T_h$ (which can belong to the essential spectrum)
in the tempered case, and $\lambda_h\in[1-\delta,1],\,\delta>0$ in the Gaussian case. 
Let $e_h\in L^2(\R^d,dx)$ satisfy $\tilde T_he_h=\lambda_h e_h$, $\parallel e_h\parallel_{L^2(\rr^d)}=1$. 
Then $e_h$ belongs to all Sobolev spaces and for all $s\in\R$
\be\label{eq:lem:regularite1}
\parallel e_h\parallel_{H^s(\R^d)}=O\Big(\Big(1+\frac{1-\lambda_h}{h^2}\Big)^{\frac s 2}\Big).
\ee
Moreover,  
\be\label{eq:lem:regularite2}
\parallel(1-\chi)(h^2\Delta)e_h\parallel_{H^s(\R^d)}=O(h^\infty)
\ee
where $\chi\in C_0^\infty(\R)$ is equal to $1$ near $0$ in the 
tempered case and $\chi=1$ on $[-s_\delta,s_\delta]$ in the Gaussian case.
\end{lem}

\bp
We use the some arguments similar to those used in \cite{LebMi}, the difference is that now we are working in $\rr^d$ instead of a compact manifold: let us write $\lambda_h=1-h^2z_h$ with $0<z_h<\kappa\gamma_d$ in the tempered case and $0<z_h<\delta h^{-2}$ in the gaussian case; and start from $(\tilde T_h-\lambda_h e_h)=0$.
Since $\tilde T_h=a_h\bT_ha_h$ it follows from Lemmas \ref{lem:prop_Gd} and \ref{lem:prop_ah} that $\tilde T_h$ is a semiclassical pseudodifferential operator on $\R^d$ of order $m\leq -1$.
In particular, it maps $L^2(\R^d)$ into $H^1(\R^d)$ and $\parallel \tilde T_h\parallel_{L^2\rightarrow H^1}=O(h^{-1})$.
Since $e_h=\frac 1{\lambda_h}\tilde T_h e_h$ and $\lambda_h$ is bounded from below, we deduce $\parallel e_h\parallel_{H^1}=O(h^{-1})$. Iterating this argument, we finally get
\be\label{eq:estim_sob_semicl}
\parallel e_h\parallel_{H^s}=O(h^{-s})
\ee
for any $s>0$ (using interpolation for non integral $s$).
Let us denote $p_h(x,\xi)$ the symbol of $\tilde T_h$. It follows from usual symbolic calculus and Lemma \ref{lem:prop_ah} that 
\be\label{eq:form_ph}
p_h(x,\xi)=a_h^2(x)G_d(\xi)+h^mr_h(x,\xi)
\ee
for some symbol $r_h\in S(\cjg \xi\cjd ^{-\max(1,\frac {d-1}2)})$ and with $m=3$ if $\rho$ is tempered and $m=2$ if $\rho$ is gaussian.

Suppose that $\rho$ is smooth tempered and let $\chi\in C_0^\infty(\R)$ be equal to $1$ near $0$.
Since $|G_d(\xi)|\leq 1$ with $G_d(\xi)\to 0$ as $|\xi|\to \infty$ and $G_d(\xi)=1 \iff \xi=0$,   
we deduce that for any cut-off function $\tilde \chi$ equal to $1$ near $0$, we have 
\[
(1-\tilde\chi(\xi))G_d(\xi)\leq (1-\epsilon)(1-\tilde{\chi}(\xi))
\]
for some $\epsilon>0$ depending on $\tilde{\chi}$.
Since $\lambda_h=1+O(h^2)$ and $a_h=1+O(h^2)$,  the symbol 
$$q_h(x,\xi)=(1-\til{\chi}(\xi))(\lambda_h-p_h(x,\xi))$$
 is bounded from below by $\frac \epsilon 2(1-\tilde\chi(\xi))$ for $h>0$ small enough. 
Moreover it is, up to a lower order symbol, equal to the symbol of $(1-\tilde{\chi}(h^2\Delta))(\la_h-\tilde{T}_h)$ and 
thus by taking $(1-\til{\chi})=1$ on the support of $(1-\chi)$, we can construct a parametrix $L_h$ with symbol $\ell_h(x,\xi)\in S(1)$ such that 
\[L_h(1-\tilde{\chi}(h^2\Delta))(\la_h-\tilde{T}_h)=(1-\chi(h^2\Delta))+h^\infty {\rm Op}_h(w_h) \]
for some symbol $w_h\in S(1)$. This clearly shows that 
\be\label{eq:local_freq}
\parallel (1-\chi(h^2\Delta))e_h\parallel_{L^2}=O(h^\infty)
\ee
and by interpolation with \eqref{eq:estim_sob_semicl} we get
\be\label{eq:local_freq_Hs}
\parallel (1-\chi(h^2\Delta))e_h\parallel_{H^s}=O(h^\infty).
\ee

It remains to show that $\chi(h^2\Delta)e_h$ is bounded in $H^s$. 
We have
\[
({\rm Op}_h(p_h)-1+h^2z_h)e_h=0.
\]
Let $b_h(x,\xi)=p_h(x,\xi)-1+h^2z_h$, then since $z_h$ is bounded, we know from (\ref{eq:form_ph}) that 
\begin{equation}\label{bhxi}
b_h(x,\xi)=a_h^2(x)G_d(\xi)-1+h^2r_h(x,\xi)
\end{equation}
for some $r_h\in S^0(1)$.
By Taylor expansion of $G_d(\xi)$ at $\xi=0$, we see that 
there exists a smooth function $F$ on $\rr^+$, strictly positive and such that $1-G_d(\xi)=\vert \xi\vert^2F(|\xi|^2)$. 
Since $a_h^2(x)=1+O(h^2)$, we get
\[
b_h(x,\xi)=-\vert \xi\vert^2F(|\xi|^2)+h^2 \tilde r_h(x,\xi)
\]
with $\tilde r_h\in S^0(1)$.
Combined with (\ref{eq:local_freq}), this shows that for any $\chi\in C_0^\infty(\R^d)$
\[
h^2\Delta F(h^2\Delta)\chi(h^2\Delta_g)e_h=O_{L^2}(h^2).
\]
Since $F$ is strictly positive on the support of $\chi$, we can construct a parametrix like above and obtain that 
\[
\parallel \chi(h^2\Delta_g)e_h\parallel_{H^2}=O(1)
\]
Iterating this process, it follows that the above bounds hold in all Sobolev spaces.\\

Consider now the case of a Gaussian density and let us prove \eqref{eq:lem:regularite2}.
For $\chi\in C_0^\infty(\R)$ equal to $1$ on $[-s_\delta, s_\delta]$ (and $0 \leq \chi\leq 1$) we get
\[
(1-\chi)(|\xi|^2)G_d(\xi)\leq (1-2\delta)(1-\chi)(|\xi|^2).
\]
Since we have $a_h\leq 1+O(h^2)$ and $\lambda_h\geq 1- \delta$ for small $h>0$, this shows that 
$(1-\chi(|\xi|^2))(\la_h-p_h(x,\xi)) \geq \frac\delta 2 (1-\chi)(|\xi|^2)$ for $h$ small and 
\eqref{eq:local_freq}, \eqref{eq:local_freq_Hs} are still valid.
Let us prove \eqref{eq:lem:regularite1}.
By definition, we have
${\rm Op}_h(b_h)e_h=0$ with $b_h(x,\xi)=a_h^2(x)G_d(\xi)-\la_h+h^2r_h(x,\xi)$ for some $r_h\in S(1)$. 
Thanks to (\ref{eq:contr_deriv_symb_gauss}),
we have $\vert\partial_x^\alpha r_h(x,\xi)\vert\leq C_\alpha \vert a_h^2(x)\vert$ for any $\alpha$
and $\vert\partial_x^\alpha (a_h^{-2}r_h(x,\xi))\vert\leq C_\alpha$.
Using again the structure of $G_d$ and dividing by $a_h^2$, it follows that 
\be\label{eq:reg_gauss0}
h^2\Delta F(h^2\Delta)e_h=(1-\lambda_ha_h^{-2}(x)+h^2{\rm Op}_h(\tilde r_h))e_h
\ee
for some symbol $\tilde r\in S(1)$.
Taking the scalar product with $\chi(h^2\Delta)^2e_h$ and using the fact that ${\rm Op}_h(\tilde r_h)$ is bounded on $L^2$, we get 
\be
\begin{split}\label{eq:reg_gauss}
\<h^2\Delta F(h^2\Delta)\chi(h^2\Delta)e_h,\chi(h^2\Delta)e_h\>&=\<(1-\lambda_ha_h^{-2}(x))\chi(h^2\Delta)e_h,\chi(h^2\Delta)e_h\>+O(h^2)\\
&=I_R(h)+J_R(h)+O(h^2)
\end{split}
\ee
where
\[\begin{gathered}
I_R(h):=\<\psi_R(x)(1-\lambda_ha_h^{-2}(x))\chi(h^2\Delta)e_h,\chi(h^2\Delta)e_h\> \\
J_R(h):=\<(1-\psi_R(x))(1-\lambda_ha_h^{-2}(x))\chi(h^2\Delta)e_h,\chi(h^2\Delta)e_h\>
\end{gathered}\]
with $\psi_R(x):=\indic_{|x|\leq R}$. Hence, it follows from (\ref{eq:ah_inv_gauss1}) that 
$I_R(h)=O(h^2R^2+1-\lambda_h)$.
On the other hand, setting $R=(1-\la_h)/(h^2\eps)$ with $\eps>0$ small enough 
but independent of $h$,  (\ref{eq:ah_inv_gauss2}) gives that
$1-\lambda_ha_h^{-2}(x)\leq -\lambda_h Ch^2|x|^2+(1-\lambda_h)<0$ if $|x|\geq R$, and hence $J_R(h)\leq 0$. 
Combined with the estimate on $I_R$, this shows that 
\[
\<h^2\Delta F(h^2\Delta)\chi(h^2\Delta)e_h,\chi(h^2\Delta)e_h\>=O(1-\lambda_h).
\]
Dividing by $h^2$ and using again the fact that $F>0$ we obtain $\|\Delta \chi(h^2\Delta)e_h\|_{L^2}=O(1+\frac{1-\lambda_h}{h^2})$. Iterating this argument and using interpolation, we obtain the desired estimates for any $H^s$.
\ep

In order to control the multiplicity of the eigenvalues  as in \cite{LebMi}, we need some compactness 
 of the family $(e_h)_h$. Since $\R^d$ is not bounded, the regularity of the eigenfunctions is not sufficient, 
 we need some decay property of the eigenfunctions near infinity. 
 For $R>0$, let $\chi_R$ be a smooth function equal to $1$ for $\vert x\vert \geq R+1$ and zero for $\vert x\vert \leq R$.
\begin{lem}\label{lem:localisation}
Let us assume that $\rho$ is tempered and let $\alpha\in]0,1[$. Suppose that $\lambda_h\in [1-\alpha h^2\frac{\kappa}{2(d+2)},1]$ and that $e_h\in L^2(\R^d,dx)$ 
satisfies $\tilde T_he_h=\lambda_h e_h$ and $||e_h||_{L^2(\rr^d)}=1$. Let $\phi\in C_0^\infty(\R)$, then there exists $R>0$ such that 
\[
\parallel \chi_R(x)\phi(h^2\Delta) e_h\parallel_{L^2(\rr^d)}=O(h^2)
\]
As a by-product, for any $s\in\R$, $\chi_R e_h$ goes to $0$ in $H^s(\rr^d)$ when $h$ goes to $0$, for any $s\geq 0$.
\end{lem}
\bp
>From the preceding Lemma, we know that 
\[
(\Delta F(h^2\Delta)+{\rm Op}_h(\til{r}_h))\phi(h^2\Delta_g)e_h=O(h^\infty).
\]
for some $\til{r}_h\in S(1)$.
On the other hand, this term can be made more precise : it follows from Lemma \ref{lem:prop_ah} and equation \eqref{eq:form_ph}
that 
\[\begin{split} 
O(h)= &\Big(-\Delta F(h^2\Delta) -\demi \gamma_d (V(x){\rm Op}_h(G_d(\xi))+ {\rm Op}_h(G_d(\xi))V(x))+z_h\Big)\phi(h^2\Delta)e_h \\
 =& -(\Delta F(h^2\Delta) +\gamma_d V(x)-z_h)e_h+ O(h^2||e_h||_{H^2})
\end{split}\]
with $\la_h=1-h^2z_h$ and using  \eqref{eq:lem:regularite1}, 
we obtain 
\[
(\Delta \til{F}(h^2\Delta)+V(x)-\tilde z_h)\tilde f_h=O(h)
\]
with $\tilde f_h:=\phi(h^2\Delta_g)e_h$, $\tilde z_h:=z_h/\gamma_d$ and $\til{F}=F/\gamma_d$.
Let $q_h(x,\xi):=\vert \xi\vert^2\til{F}(h\xi)+ V(x)-\tilde z_h$. Since $F\geq 0$, 
it follows from assumption (\ref{hyp2}) that there exists $R>0$ such that for any $\xi\in\R^d$ and any $\vert x\vert\geq R$, we have $q_h(x,\xi)\geq (1- \alpha)\kappa/2$ if $1-\la_h\leq \alpha \kappa h^2/2(d+2)$. 
Hence we can build a parametrix for $q_h$ on the support of $\chi_R$ and 
this shows that $\parallel \chi_R\tilde f_h\parallel_{L^2}=O(h)$.
Using interpolation and the fact that $(e_h)$ is bounded in $H^s$, we obtain directly the same bounds in $H^s$.
\ep

\begin{lem}\label{lem:estim_eh_inf} Suppose that $\rho$ is Gaussian. Let $\delta>0$ and $\chi\in C_0^\infty(\R)$ be equal to $1$ on $[- s_\delta,s_\delta]$, then there exists $h_0$ such that, for  any $k,s\in\N$ there exists $C_{k,s}>0$ such that for all $h\leq h_0$ and any eigenfunction
 $e_h\in L^2(\R^d)$ of $\til{T}_h$ with eigenvalue $\lambda_h\in[1-\delta,1]$, we have
\be\label{eq:dec_fct_ppre}
\|\cjg x\cjd^k\chi(h^2\Delta)e_h\|_{H^s(\R^d)}\leq C_{k,s}\|\chi(h^2\Delta)e_h\|_{H^{s+k}(\R^d)}
\ee
\end{lem}
\bp 
It follows from \eqref{eq:reg_gauss0} and \eqref{eq:lem:regularite2} that 
\be\label{eq:minorJR}
(1-\lambda_ha_h^{-2}(x))\chi(h^2\Delta)e_h=h^2{\rm Op}_h(r_h)\chi(h^2\Delta)\Delta e_h
\ee
for some $r_h\in S(1)$.
Let $R>0$ be sufficiently large so that $a_h^{-2}(x)\geq 1+Ch^2|x|^2$ for $|x|\geq R$. Then, if $\la_h=1-h^2z_h$, one has 
for $|x|>R$
\begin{equation}\label{ahmoins2}
-1+\lambda_ha_h^{-2}(x)\geq h^2(C|x|^2-z_h)\geq C'h^2(1+|x|^2)
\end{equation}
for some $C'>0$ independent of $h$. 
We take $\psi_R\in C_0^\infty(\R^d)$ be equal to $1$ for $|x|\geq R+1$ and $0$ for $|x|\leq R$ , then
by \eqref{ahmoins2} and  \eqref{eq:ah_symb_gauss}, we deduce easily that 
$\cjg x\cjd^2h^{-2}(-1+\lambda_ha_h^{-2})^{-1}\psi_R\in S(1)$ and therefore
\[
\cjg x\cjd^2\psi_R(x)\chi(h^2\Delta)e_h={\rm Op}_h(\til{r}_h)\chi(h^2\Delta) \Delta e_h
\]
for some $\tilde r_h\in S(1)$.
Therefore, for any $s\geq 0$, we have
\[
\|\<x\>^2\chi(h^2\Delta)e_h\|_{H^s(\R^d)}\leq C \|\chi(h^2\Delta)e_h\|_{H^{s+2}(\R^d)}
\]
Iterating this argument $k/2$ times and using \eqref{eq:lem:regularite1}, we get  \eqref{eq:dec_fct_ppre}.
\ep

\section{Proof of Theorem \ref{th:analyse_spec_temp} and \ref{th:analyse_spec_gauss}}

\subsection{Spectrum localisation}
We work as in \cite{LebMi}  and we only give a sketch of the proof since it is rather similar. 
The main difference with the situation in \cite{LebMi} 
is that we work on unbounded domains, so that Sobolev embedding do not provide directly compactness.
In both tempered and Gaussian case, we will use the following observation:  suppose that $\varphi$ is a smooth function, then it follows from Lemma \ref{lem:prop_ah} and the expansion $G_d(\xi)=1-\gamma_d|\xi|^2+O(|\xi|^4)$ as $|\xi|\to 0$
that 
\be\label{eq:dvt_taylor_Th}
\frac{1-\tilde T_h}{h^2}\varphi=\gamma_dL_\rho \varphi+h^2\psi
\ee
where  $\|\psi\|_{L^2(\R^d)}=O(\|\varphi\|_{H^4(\R^d)})$ in the tempered case and $\|\psi\|_{L^2(B(0,Mh^{-1}))}=O(\||x|^4\varphi\|_{H^4(B(0,Mh^{-1}+1))})$  for any $h$-independent $M>0$ in the Gaussian case.

We start with the case of a tempered density and 
follow the strategy of \cite{LebMi}. Since $T_h$ and $\tilde T_h$ are unitarily conjugated by $\Omega:L^2(\R^d, dx)\rightarrow L^2(\R^d, d\nu_h)$, the eigenvalues of $T_h$ on $L^2(\R^d, d\nu_h)$ (and their multiplicities) are exactly those of $\tilde T_h$ on $L^2(\R^d, dx)$.\\
First, assume that $(L_\rho-\mu)e=0$ for some $\mu\in [0,\kappa)$ and $e\in H^2(\R^d)$, $\parallel e\parallel_{L^2(dx)}=1$. Then, $e$ is in fact in $C^\infty$ and using \eqref{eq:dvt_taylor_Th} with $\varphi=e$, we get easily
\[
\frac{1- \tilde T_h}{h^2} e=\gamma_d\mu  e+O_{L^2}(h^2).
\]
Since $\tilde T_h$ is self-adjoint, this shows that ${\rm dist}(\gamma_d\mu,\sigma(\Delta_h))=O(h^2)$ with
\[\Delta_h:=\frac{1-\til{T}_h}{h^2},\] 
and that there exist $C_0>0, C_1>0,h_0>0$ such that for all $0<h\leq h_0$ and $\mu\in \sigma(L_\rho)\cap [0,\kappa-C_1h^2)$, 
the number of eigenvalues of $\Delta_h$ in 
$[\gamma_d\mu-C_0h^2,\gamma_d\mu+C_0h^2]$ is bounded below by the multiplicity of $\mu$.

Conversely, consider an eigenfunction $e_h$ of $\Delta_h$ corresponding 
to an eigenvalue $z_h\in [0,\gamma_d\kappa)$, then 
using Lemma (\ref{lem:regularite}), we get 
 \[
z_h e_h=\Delta_h e_h=\gamma_dL_\rho e_h+O_{L^2}(h^2).
 \]
 This shows that all the eigenvalues of  $\Delta_h$ are at distance at most $Ch^2$ of the spectrum of $\gamma_d L_\rho$. Let us now consider an orthonormal set of eigenfunctions $e^j_h$ of $\Delta_h$ associated to the eigenvalues $z^j_h$ contained in 
$ [\gamma_d\mu-C_0h^2,\gamma_d\mu+C_0h^2]$ for some $\mu\in \sigma(L_\rho)\cap [0,\alpha\kappa]$, 
where with $C_0,C_1$ are the constants given above.
Let $R>0$ be fixed as in Lemma \ref{lem:localisation}. From Lemmas \ref{lem:localisation} and \ref{lem:regularite} , each eigenfunction can be decomposed as
\[
e^j_h=u^j_h+v^j_h
\]
with $u^j_h$ bounded in any $H^s$ and  supported in $B(0,R)$and $v^j_h$ converging to $0$ in $H^s$ when $h$ goes to $0$.
Since $H^s(B(0,R))$ is compactly embedded in $H^2$ for $s$ larger than $2$, we can assume 
(extracting a subsequence if necessary) that the $e^j_h$ converge to some $f^j$ in 
$H^2(\R^d,dx)$ and $z^j_h$ converges to $\mu/\gamma_d$.
Hence, the $(f^j)_j$ provide an orthonormal family of eigenfunctions 
of $L_\rho$ associated to the eigenvalue $\mu$. 
This shows that the number of eigenvalues of $\Delta_h$ in $[\gamma_d\mu-C_0h^2,\gamma_d\mu+C_0h^2]$  is exactly
the multplicity of $\mu$ as an eigenvalue of $L_\rho$,
and achieves the proof of Theorem \ref{th:analyse_spec_temp}.

Notice in particular that our proof does not rule out the possibility of an infinite sequence of eigenvalues $z^j_h$
for $\Delta_h$ converging to the bottom of the essential spectrum $\kappa$.\\

Assume now that $\rho$ is Gaussian and start with $(L_\rho-\mu)e=0$  with $\|e\|_{L^2}=1$. 
It follows from \eqref{eq:dvt_taylor_Th} that 
\[
\Delta_he=\indic_{|x|<h^{-1}}\gamma_dL_\rho e+\indic_{|x|\geq h^{-1}}\Delta_he+h^2\psi
\]
with $\psi$ supported in $B(0,Mh^{-1})$ and $\|\psi\|_{L^2}=O(\|\cjg x\cjd ^4e\|_{H^4(B(0,Mh^{-1}+1))})$. Since $e=p(x)e^{-\alpha|x|^2}$ for some polynomial $p$, then 
$\|\psi\|_{L^2}$ is bounded uniformly with respect to $h$. The same argument and  
$\indic_{|x|\geq h^{-1}}\Delta_h=\indic_{|x|\geq h^{-1}}\Delta_h\indic_{|x|\geq h^{-1}-h}$ shows that 
$\|\indic_{|x|\geq h^{-1}}\Delta_he\|_{L^2}=O(h^{-2}e^{-ch^{-2}})$.
This implies that 
\[
\Delta_h e=\gamma_d\mu  e+O_{L^2}(h^2)
\]
Like in the tempered case, it follows that  ${\rm dist}(\gamma_d\mu,\sigma(\Delta_h))=O(h^2)$ and that for any given $L>0$ there exists $C_0>0,h_0>0$ such that for all $0<h\leq h_0$ and all $\mu\in\sigma(L_\rho)$ with $\mu \leq L$, the number of eigenvalues of $\Delta_h$ in 
$[\gamma_d\mu-C_0h^2,\gamma_d\mu+C_0h^2]$ is bounded by the multiplicity of $\mu$.

Conversely, suppose now that $\tilde T_he_h=(1-h^2\gamma_dz_h)e_h$ for some $e_h\in L^2(\rr^d)$ such that $\parallel e_h\parallel_{L^2}=1$ and $z_h\in[0,L]$, $L>0$ being fixed. From Lemmas \ref{lem:regularite} and \ref{lem:estim_eh_inf}, we know that
\[
z_h e_h=\Delta_he_h=L_\rho e_h+O_{L^2}(h^2),
\]
this shows that the distance 
of the eigenvalues of  $\Delta_h$ (less than $L$) 
to $\sigma(L_\rho)$ is of order $O(h^2)$.

To get the equality between the multiplicities, we work as in the tempered case and consider an orthonormal family
of eigenfunctions $e^j_h$ of $\Delta_h$ associated to the eigenvalues $z^j_h$ contained in 
$[\gamma_d\mu-C_0h^2,\gamma_d\mu+C_0h^2]$.
It follows from Lemmas \ref{lem:localisation} and \ref{lem:estim_eh_inf} that 
\[
e^j_h=u^j_h+O(h^\infty)
\]
with $u^j_h:=\chi(h^2\Delta)e^j_h$ bounded uniformly with respect to $h$ in $ \cjg x\cjd^{-k}H^s(\R^d) $ for any $k,s\geq 0$. 
Then the family $(u^j_h)_{h>0}$ is compact in $H^2(\rr^d)$ and extracting a subsequence if necessary, we can then assume that
both $u^j_h$ and $e^j_h$ converge to some $f^j$ in $H^2$ and $z_h$ converges to $z\in[0,L]$.
We split $u^j_h$ into $\psi_h(x)u^j_h+(1-\psi_h(x))u^j_h$ where $\psi_h$ is smooth, supported in $|x|\leq 1/h$ and equal to $1$
in $|x|\leq 1/2h$. In particular we have that $||(1-\psi_h)u^j_h||_{H^4}= O(h^\infty)$. 
On the other hand, it follows from \eqref{eq:dvt_taylor_Th} that 
\[
\begin{split}
z^j_h e^j_h&=\Delta_h e^j_h=\Delta_h (\psi_h u^j_h)+O(h^\infty)\\
&=\gamma_d L_\rho (\psi_h u^j_h)+O(h^2\|\cjg x\cjd^4 \psi_hu^j_h\|_{H^4})+O(h^\infty)\\
&= \gamma_d L_\rho (e^j_h)+O(h^2\|\cjg x\cjd^4 e^j_h\|_{L^2})+O(h^\infty)\\
z^j_h e^j_h&=\gamma_d L_\rho(e^j_h)+O(h^2\| e^j_h\|_{L^2 })+ O(h^\infty)
\end{split}
\]
where we used Lemma \ref{lem:estim_eh_inf} in the last line.
Making $h\rightarrow 0$, we show that $(f^j)_j$ is an orthonormal family of eigenfunctions of $L_\rho$ 
associated to the eigenvalue $z=\mu/\gamma_d$.
This achieves the proof of \eqref{eq:th:loc_vp}. \\

\subsection{The weyl estimate}
It remains to prove the Weyl estimate on the number of eigenvalues in the Gaussian density case.
Fix $\delta>0$ small, then for $\tau>0$, let us define the operator on $\rr^d$
\[P_\tau =\tau(\chi^2(\sqrt{\Delta/\tau})+\chi^2(\sqrt{|x|^2/\tau}))\]
where $\chi\in C^\infty((0,\infty))$ is a positive increasing function which satisfies $\chi(x)=x$ for $x<1-\delta$ and  $\chi(x)=1$ for $x>1$.  
Clearly $P_\tau$ is a self-adjoint bounded operator on $L^2(\rr^d)$ with norm less or equal to $2\tau$ and 
since for any function $f\in L^2$ such that $f$ is supported in $|x|>\tau$ or $\hat f$ is supported in $\vert \xi\vert>\tau$, one has $\cjg P_\tau f,f\cjd\geq \tau||f||_{L^2}^2$, the essential spectrum    
is contained in the interval $[\tau,2\tau]$. Let $\Pi_{\tau/2}=\indic_{[0,\tau/2]}(P_\tau)$ 
be the orthogonal spectral  projector, it is then finite rank by what we just said.  
For $f$ in the range of $1-\Pi_{\tau/2}$, we shall prove that there is $\eps>0$, $C>0$ independent of $\tau,h$ such that for 
$\tau\leq \eps h^{-2}$
\begin{equation}\label{estimationweyl}
\cjg T_hf,f\cjd \leq (1-C\tau h^2) ||f||_{L^2}^2 .
\end{equation}  
Notice that if $(1-\Pi_{\frac\tau 2})f=f$, we have $\cjg P_{\tau}f,f\cjd\geq \demi \tau||f||^2_{L^2}$ and thus 
\begin{equation}\label{xalpha/2} 
||\chi(\sqrt{\Delta/\tau})f||^2+||\chi(\sqrt{|x|^2/\tau})f||^2\geq \demi||f||^2 .
\end{equation}
We first assume that $||\chi(\sqrt{|x|^2/\tau})f||^2\geq \frac{1}{4}||f||^2$, then using that $\bar{T}_h$ has $L^2\to L^2$ norm bounded by $1$ we deduce
\[\cjg a_h\bar{T}_ha_hf,f\cjd=\cjg \bar{T}_ha_hf,a_hf\cjd \leq ||a_h f||_{L^2}^2.\]
But from \eqref{eq:ah_inv_gauss1} and \eqref{eq:ah_inv_gauss2}, we also have that there is $\eps>0$, $C>0$ independent of $\tau,h$ such that if $\tau\leq \eps h^{-2}$, 
\[ a_h^2(x)\leq 1-Ch^2\tau \chi(\sqrt{|x|^2/\tau})^2.\]  
Thus we obtain by combining with \eqref{xalpha/2}
\begin{equation}\label{firstbound}
\cjg a_h\bar{T}_ha_hf,f\cjd\leq (1-Ch^2\tau/4) ||f||_{L^2}^2.
\end{equation}
Assume now that \eqref{xalpha/2} is not true, then since $(1-\Pi_{\tau/2})f=f$ this implies that 
\begin{equation}\label{xialpha/2} 
||\chi(\sqrt{\Delta/\tau})f||^2\geq \frac{1}{4}||f||^2 
\end{equation}
and we shall prove that \eqref{firstbound} holds as well in that case. Using $a_h^2\leq 1+Ch^2$ for some $C>0$, 
let us write for $f\in L^2$
\begin{equation}\label{casxi}
\begin{split}
\cjg a_h\bar{T}_ha_hf,f\cjd= &\cjg a_h^2\bar{T}_hf,f\cjd+ \cjg a_h[\bar{T}_h,a_h]f,f\cjd \\
\leq & (1+Ch^2)||\bar{T}_hf||_{L^2} \, ||f||_{L^2}+ \cjg a_h[\bar{T}_h,a_h]f,f\cjd.
\end{split} \end{equation}
Using the fact that $\bar{T}_h=G_d(hD_x)$ is a semiclassical pseudo-differential operator  with symbol $G_d\in S(1)$
defined in \eqref{eq:def_ah} and the estimates  $|\pl_x^\alpha a_h|=O(h)$ if $|\alpha|>0$ of Lemma \ref{lem:prop_ah}, 
we deduce from the composition law of semiclassical pseudo-differential operators that  $[\bar{T}_h,a_h]=h^2{\rm Op}_h(c_h)$ where $c_h\in S(1)$ is a uniformly bounded symbol in $h$. Therefore by Calder\'on-Vaillancourt theorem, $||a_h[\bar{T}_h,a_h]||_{L^2\to L^2}=O(h^2)$ and thus 
\begin{equation}\label{commut}
\cjg a_h[\bar{T}_h,a_h]f,f\cjd\leq Ch^2||f||^2_{L^2}
\end{equation}
for some $C>0$ uniform in $h$ and independent of $\tau$.
Now using Plancherel, $\cjg \bar{T}^2_hf,f\cjd=\int_{\rr^d}G_d^2(h\xi)|\hat{f}(\xi)|^2d\xi$ where $G_d$
is defined in \eqref{eq:def_ah}. Now since $G_d(\xi)\to 0$ as $\xi\to \infty$ and $G^2_d(\xi)\leq 1-C|\xi|^2$ for some 
$C$ when $\xi$ is small, we directly obtain that there is $\eps>0$ independent of $\tau,h$ such that if $\tau\leq \eps h^{-2}$, the bound
$G^2_d(h\xi)\leq 1-Ch^2\tau \chi(\sqrt{|\xi|^2/\tau})^2$. Combined with \eqref{xialpha/2}, this implies that 
$||\bar{T}_hf||_{L^2}\leq (1-Ch^2\tau)||f||_{L^2}$ and thus, by combining this with \eqref{casxi}  and  
\eqref{commut}, \eqref{firstbound} holds if $\tau>\tau_0$ for some $\tau_0>0$ independent of $h$  and we 
have proved \eqref{estimationweyl}.

By the mini-max principle, one deduces from \eqref{estimationweyl} that the number of eigenvalues 
of $T_h$ in $[1-Ch^2\tau,1]$ counted with multiplicites  is bounded by the rank of $\Pi_{\tau/2}$.
Now, to prove the Weyl estimate \eqref{eq:weyl}, it remains to show that  
${\rm Rank}(\Pi_{\tau/2})=O(\tau^{d})$.  This is a rather standard result (see for instance \cite[page 115]{DiSj99} 
for a comparable estimate), but we write some details

Let us consider $\hbar:=1/\sqrt{\tau}$ as a semiclassical parameter. The operator
$P(\hbar):=\hbar^{2}P_{\hbar^{-2}}$ is a $\hbar$ semi-classical operator with a symbol in the 
class $S(1)$ given by $p_\hbar(x,\xi)=\chi^2(|\xi|)+\chi(\hbar|x|)$, more precisely 
$P(\hbar)$ is the Weyl quantization of the symbol $p_\hbar(x,\xi)$. 
Let $f\in C_0^\infty(\R)$ be such that $f(s)=1$ for $|s|\leq 1$, $f(s)=0$ for $|s|\geq 2$ 
and $0\leq f\leq 1$. Consider the harmonic oscillator on $\R^d$, $H=\Delta+|x|^2$ and 
define the operator 
\be\label{eq:def_Kh}
\Pi^H_\hbar=f(\hbar^2H).
\ee
Then $\Pi^H_\hbar$ is a non-negative self-adjoint operator, it is bounded by $1$, it has finite rank 
and ${\rm rank}(\Pi^H_\hbar)=O( \hbar^{-d})$.
>From the min-max principle, to prove a Weyl estimate for $P(\hbar)$, 
it suffices to show that  for all $u\in L^2$
 \begin{equation}\label{minmax}
\cjg P(\hbar) u,u\cjd +\cjg \Pi^H_\hbar u,u\cjd \geq c
 \end{equation} 
for some $c>0$.
First, we claim that the operator $\Pi^H_\hbar$ can be written under the form 
\begin{equation}\label{claim}
\Pi^H_\hbar={\rm Op}_\hbar(f(\hbar^2|x|^2+|\xi|^2))+ R_\hbar, \,\textrm{ where }||R_\hbar||_{L^2\to L^2}=O(\hbar).
\end{equation}
Let $\Omega$ be a fixed compact subset of $\cc$ whose intersection with $\rr$ contains ${\rm supp}(f)$. 
Then, it is easy to check that for all $s\in\Omega\cap (\cc\setminus \rr)$ 
\[(\hbar^2H-s){\rm Op}_\hbar\Big(\frac{1}{\hbar^2|x|^2+|\xi|^2-s}\Big)= 1+ \hbar{\rm Op}_\hbar(q_\hbar(x,\xi;s))\]
for some symbol $q_\hbar(x,\xi; s)\in S(1)$, satisfying for any $\alpha,\beta$ 
\[ |\pl_x^\alpha\pl_\xi^\beta q_\hbar(x,\xi)|\leq C_{\alpha,\beta}|{\rm Im}(s)|^{-3-|\alpha|-|\beta|}\]
for some $C_{\alpha,\beta}$ uniform in $h,s$. Then this implies 
\[ (\hbar^{2}H-s)^{-1}={\rm Op}_\hbar\Big(\frac{\la}{\hbar^2|x|^2+|\xi|^2-s}\Big)-\hbar (\hbar^{2}H-s)^{-1}{\rm Op}_\hbar(q_\hbar(x,\xi;s))\]
but by the Calderon-Vaillancourt theorem and the spectral theorem for $H$, we deduce that 
\begin{equation}\label{h^2H-s}
(\hbar^{2}H-s)^{-1}={\rm Op}_\hbar\Big(\frac{1}{\hbar^2|x|^2+|\xi|^2-s}\Big)+\hbar W_\hbar
\end{equation} 
for some bounded operator $W_\hbar$ on $L^2$ with norm $O(\hbar|{\rm Im}(s)|^{-N})$ for some $N$ depending only on the dimension 
$d$. It remains to apply Helffer-Sj\"ostrand formula \cite[Th 8.1]{DiSj99} 
with $\til{f}\in C_0^\infty(\Omega)$ an almost analytic extension of $f$
\[f(\hbar^2H)=\frac{1}{2i\pi}\int_\cc \bar{\pl}\til{f}(s)(\hbar^2H-s)^{-1}ds\wedge d\bar{s} \]
and we deduce directly \eqref{claim} from \eqref{h^2H-s}.
Observe that the symbol of $P(\hbar)+\Pi^H_\hbar$ satisfies that there exists $C>0$ such that 
\[ \chi^2(|\xi|)+\chi^2(\hbar|x|)+f(\hbar^2|x|^2+|\xi|^2) \geq C\] 
for all $0<\hbar\leq \hbar_0$. Therefore, by G\"arding inequality, \eqref{minmax} is satisfied for some $c>0$, and using 
the min-max principle, this implies easily that the number of eigenvalues of $P(\hbar)$ 
less or equal to $C/2$ is bounded above by ${\rm rank}(\Pi^H_\hbar)=O(\hbar^{-d})$, and 
this conclude the proof of the Weyl estimate for $T_h$.

\section{Convergence to stationarity}In this section, we study the convergence of the iterated kernel $T_h^n(x,dy)$ towards its stationnary measure d$\nu_h$ when $n$ goes to infinity. The measure $d\nu_h$ is associated to the orthogonal projection $\Pi_{0,h}$ onto  constant functions in $L^2(d\nu_h)$:
\be
\Pi_{0,h}(f)=\int_{\R^d}f(y)d\nu_h(y)
\ee
The following proposition gives a convergence result in $L^2$ norm.
\begin{prop}\label{prop:conv_L2} Let $\alpha>0$ be fixed.
There exists $C>0$ and $h_0>0$ such that for all $h\in]0,h_0]$ and all $n\in\N$, we have
\be
\parallel T_h^n-\Pi_{0,h}\parallel_{L^2(d\nu_h)\rightarrow L^2(d\nu_h)}\leq Ce^{-nh^2\min(\mu_1,(1-\alpha)\kappa)}.
\ee
\end{prop}
\bp
This is a direct consequence of the spectral theorem and Theorems \ref{th:analyse_spec_temp},  \ref{th:analyse_spec_gauss}. 
\ep

Let us now introduce the total variation distance, which is much stronger than the $L^2$ norm. If $\mu$ and $\nu$ are two probability measures on a set $E$, their total variation distance is defined by 
\begin{equation*}
\|\mu-\nu\|_{TV}=\sup_{A}|\mu(A)-\nu(A)|
\end{equation*}
where the sup is taken over all measurable sets. Then, a standard computation shows that
\begin{equation*}
\|\mu-\nu\|_{TV}=\frac 1 2\sup_{\parallel f\parallel_{L^\infty=1}}\vert\mu(f)-\nu(f)\vert
\end{equation*}

The following theorem shows that the convegence in total variation distance can not be uniform with respect to the starting point $x$. This has to be compared with the results in the case of compact state space \cite{DiaLeb}, \cite{DiaLebMi} and \cite{LebMi} where the convergence is uniform in $x$.
\begin{thm}\label{th:non_unif_TV}
There exists $C>0$ such that  for any $n\in\N$, $h\in ]0,1]$, $\tau>0$ and $|x|\geq \tau+(n+1)h$, we have
\be
\|T_h^n(x,dy)-d\nu_h\|_{TV}\geq 1-Cp(\tau)
\ee
where
$p(\tau)=e^{-2\alpha\tau(\tau-h)}$ if $\rho=\beta e^{-\alpha |x|^2}$ is Gaussian and 
$p(\tau)=\int_{|y|\geq\tau} \rho(y)^2dy$ if $\rho$ is tempered.
\end{thm}

\bp
Let $\tau>0$ and $n\in\N$. Consider the function 
\be
\begin{split}
f_\tau(x)&=\indic_{[\tau,+\infty[}(|x|)-\indic_{[0,\tau[}(|x|)\\
&=-1+2\indic_{[\tau,+\infty[}(|x|).
\end{split}
\ee
For $x\in\R^d$ such that $|x|\geq \tau+(n+1)h$, thanks to finite speed propagation we have
\be\label{eq:const_inf}
T_h^nf_\tau(x)=1.
\ee
On the other hand, we also have
\be
\begin{split}
\Pi_{0,h}f_\tau&=\int_{\R^d}f_\tau(y)d\nu_h(y)=-1+2\int_{|y|\geq \tau}d\nu_h(y)\\
&=-1+\frac 2 {Z_h}\int_{|y|\geq \tau}m_h(y)\rho(y)dy
\end{split}
\ee
If $\rho$ is tempered, then $m_h(y)\leq Ch^d\rho(y)$ for some constant $C>0$. Hence, $\Pi_{0,h}f_\tau\leq -1+Cp(\tau)$ with $p(\tau)=\int_{|y|\geq\tau} \rho(y)^2dy$.
Combined with \eqref{eq:const_inf}, this shows the anounced result in the tempered case.

Suppose now that $\rho(x)=\beta e^{-\alpha |x|^2}$ is Gaussian for some $\alpha,\beta>0$. 
Then $m_h(y)\leq Ch^de^{-\alpha |y|^2+2h\alpha |y|}$ for any $h\in]0,1]$. 
Hence,
\be
\Pi_{0,h}f_\tau\leq -1+C\int_{|y|\geq\tau}e^{-2\alpha(|y|^2-h|y|)}dy\leq -1+C p(\tau)
\ee
with $p(\tau)=e^{-2\alpha\tau(\tau-h)}$.
Using again \eqref{eq:const_inf}, this shows the anounced result in the Gaussian case.
\ep

In the following theorem, $g(h)=1-\lambda_1(h)$ denotes the spectral gap of $T_h$, whose asymptotics is given in Theorems \ref{th:analyse_spec_temp} and \ref{th:analyse_spec_gauss}.

\begin{thm}\label{th:TV_estim}
 There exists $C>0$ and $h_0>0$ such that  for any $n\in\N$, $h\in ]0,h_0]$, $\tau>0$, 
\be
\sup_{|x|<\tau}\|T_h^n(x,dy)-d\nu_h\|_{TV}\leq C q(\tau,h)e^{-ng(h)}
\ee
where $q(\tau,h)=e^{\alpha\tau(\tau+3h)}$ if $\rho=\beta e^{-\alpha|x|^2}$ is Gaussian and $q(\tau,h)=h^{-\frac d 2}\sup_{|x|<\tau}\frac{1}{\rho(x)}$ if $\rho$ is tempered.
\end{thm}

\bp Assume that $h_0>0$ is such that the results of the previous section hold true for $h\in]0,h_0]$.
Observe that 
\be
\begin{split}
\sup_{|x|\leq \tau}\|T_h^n(x,dy)-d\nu_h\|_{TV}&=\frac 1 2 \sup_{|x|\leq \tau}\sup_{\|f\|_{L^\infty}=1}|T_h^nf(x)-\Pi_{0,h}f|\\
&=\frac 1 2\|T_h^n-\Pi_{0,h}\|_{L^\infty(\R^d)\rightarrow L^\infty(|x|\leq\tau)}
\end{split}
\ee
Suppose first that $\rho$ is tempered and denote $B_\tau$ the ball of radius $\tau$ centred in $0$ and 
$I_n(\tau,h)=\|T_h^n-\Pi_{0,h}\|_{L^\infty(\R^d)\rightarrow L^\infty(B_\tau)}$ . Then, denoting $L^2(d\nu_h)$ for $L^2(\rr^d,d\nu_h)$,
\be
\begin{split}
I_n(\tau,h)&\leq \|T_h\|_{L^2(d\nu_h)\rightarrow L^\infty(B_\tau)}\|T_h^{n-2}-\Pi_{0,h}\|_{L^2(d\nu_h)\rightarrow L^2(d\nu_h)}\|T_h\|_{L^\infty(\rr^d)\rightarrow L^2(d\nu_h)}\\
&\leq \|T_h\|_{L^2(d\nu_h)\rightarrow L^\infty(B_\tau)} e^{-(n-2)g(h)}
\end{split}
\ee
where we have used Proposition \ref{prop:conv_L2} and the fact that $\|T_h\|_{L^\infty(\R^d)\rightarrow L^2(d\nu_h)}=1$.
To estimate $T_h$ from $L^2(d\nu_h)$ into $L^\infty(B_\tau)$ we consider $f\in L^2(d\nu_h)$ such that $\|f\|_{L^2(d\nu_h)}=1$. Then,
\be\label{eq:estim_Th_Linf}
\begin{split}
|T_hf(x)|&\leq \frac 1 {m_h(x)}(\int_{|x-y|<h}\frac {Z_h^2}{m_h(y)^2}d\nu_h)^{\frac 1 2}\\
&\leq\frac {Z_h^{\frac 1 2}} {m_h(x)}(\int_{|x-y|<h}\frac {\rho(y)}{m_h(y)}dy)^{\frac 1 2}
\end{split}
\ee
Since $\rho$ is tempered we have $m_h(z)\geq Ch^d\rho(z)$ for some $C>0$ and we deduce from the above estimate that
$|T_hf(x)|\leq C/(h^{\frac d 2}\rho(x))$. Taking the supremum over $x\in B_\tau$ we obtain the announced result in the tempered case.

Suppose now that $\rho=\beta e^{-\alpha|x|^2}$ is Gaussian. Since $T_h$ is Markov and $g(h)$ is of order $h^2$, we can assume $n>h^{-2}$. For $k\in\N$ let $\sigma_k(h)=\frac {1-\lambda_k(h)}{h^2}$, where $1=\lambda_0(h)>\lambda_1(h)\geq\lambda_2(h)\geq\ldots\geq\lambda_k(h)$ denote the eigenvalues of $T_h$. Denote also $e_{k,h}$ the eigenvector associated to $\lambda_k(h)$ normalized in $L^2(d\nu_h)$ and 
$\Pi_{k,h}=\<.,e_{k,h}\>_{L^2(d\nu_h)}e_{k,h}$ the associated projector. We write the eigenvalues under the form $\la_k(h)=1-h^2\sigma_k(h)$, then the spectral gap $g(h)=h^2\sigma_1(h)$. 
Let $\delta>0$ and decompose $T_h=T_{h,1}+T_{h,2}$ with
\be
T_{h,1}=\sum_{\sigma_1(h)\leq\sigma_k(h)<(1-\delta) h^{-2}}(1-h^2\sigma_k(h))\Pi_{k,h}
\ee
>From the spectral theorem, we deduce that $\parallel T_{h,2}^{n-1}\parallel_{L^2\rightarrow L^2}\leq C (1-\delta)^n$.  
On the other hand, for $\rho$ gaussian, we have $m_h(z)\geq Ch^d\rho(z)e^{-2h\alpha\vert z\vert}$. Combining this estimate with 
\eqref{eq:estim_Th_Linf}, we get
\be
\parallel T_h\parallel_{L^2(\R^d)\rightarrow L^\infty(B_\tau)}\leq Ch^{-\frac d 2}e^{\alpha\tau(\tau+3h)}
\ee
Since $T_{h,2}^n=T_hT_{h,2}^{n-2}T_h$, we can combine this with the $L^2$ estimate, to get
\be
\parallel T_{h,2}^n\parallel_{L^\infty(\R^d)\rightarrow L^\infty(B_\tau)}\leq Ch^{-\frac d 2}e^{\alpha \tau(\tau+3h)}(1-\delta)^n\leq q(\tau,h)e^{-ng(h)}
\ee
since $h^{-\frac d 2}(1-\delta)^n\ll e^{-ng(h)}$. Hence, it remains to study $T_{h,1}^n$.

Since $d\nu_h$ is a probability, then
\be
\parallel \Pi_{k,h}\parallel_{L^\infty(\R^d)\rightarrow L^\infty(B_\tau)}\leq \parallel e_{k,h}\parallel_{L^\infty(B_\tau)} \parallel e_{k,h}\parallel_{L^1(d\nu_h)}\leq \parallel e_{k,h}\parallel_{L^\infty(B_\tau)}
\ee
>From Lemma \ref{lem:regularite} ans Sobolev embedding, we know that $\parallel\Omega^*e_{k,h}\parallel_{L^\infty(\R^d)}\leq C\sigma_{k,h}^{\frac d 2}$. Hence,
\be
\parallel \Pi_{k,h}\parallel_{L^\infty(\R^d)\rightarrow L^\infty(B_\tau)}\leq \sup_{B_\tau}(\frac {Z_h}{m_h(x)\rho(x)})^{\frac 1 2}\parallel\Omega^*e_{k,h}\parallel_{L^\infty(\R^d)}\leq C\sigma_{k,h}^{\frac d 2}e^{\alpha\tau(\tau+3h)}
\ee
Using this estimate we get immediatly
\be
\parallel T_{h,1}^n\parallel_{L^\infty(\R^d)\rightarrow L^\infty(B_\tau)}\leq C e^{\alpha\tau(\tau+3h)}
 \sum_{\sigma_1(h)\leq\sigma_k(h)<(1-\delta) h^{-2}}(1-h^2\sigma_k(h))^n\sigma_{k,h}^{\frac d 2}
\ee
Using the Weyl estimate \eqref{eq:weyl} and the same argument as in \cite{LebMi}, we get
\be
\parallel T_{h,1}^n\parallel_{L^\infty(\R^d)\rightarrow L^\infty(B_\tau)}\leq  C e^{\alpha\tau(\tau+3h)}\int_{\sigma_{1,h}}^\infty (1+x)^N e^{-nh^2x}dx\leq Ce^{\alpha \tau(\tau+3h)}e^{-nh^2\sigma_{1,h}}
\ee
for some $N>0$. This completes the proof in the Gaussian case.
\ep

\end{document}